\documentclass{amsart}

\usepackage{amsfonts}
\usepackage{amssymb}
\usepackage{amsmath}
\usepackage{amsthm}
\usepackage{amsmath,amscd}

\usepackage[OT2,T1]{fontenc}
\DeclareSymbolFont{cyrletters}{OT2}{wncyr}{m}{n}
\DeclareMathSymbol{\Sha}{\mathalpha}{cyrletters}{"58}

\newtheorem{theorem}{Theorem}[section]
\newtheorem{lemma}[theorem]{Lemma}
\newtheorem{prop}[theorem]{Proposition}
\newtheorem{cor}[theorem]{Corollary}
\newtheorem{conjecture}[theorem]{Conjecture}

\theoremstyle{definition}
\newtheorem{definition}[theorem]{Definition}

\newtheorem{remark}[theorem]{Remark}

\newcommand{\Z}{\mathbb{Z}}

\newcommand{\Q}{\mathbb{Q}}

\newcommand{\Qb}{\bar{\mathbb{Q}}}
\newcommand{\C}{\mathbb{C}}
\newcommand{\Cx}{\mathbb{C}^\times}
\newcommand{\F}{\mathbb{F}}

\newcommand{\Kx}{K^\times}
\newcommand{\Lx}{L^\times}
\newcommand{\kx}{k^\times}
\newcommand{\kb}{\bar{k}}

\newcommand{\braces}[1]{\left\{#1\right\}}

\newcommand{\Sk}{\mathfrak{S}}

\newcommand{\Ic}{\mathcal{I}}

\newcommand{\Kc}{\mathcal{K}}

\renewcommand{\Mc}{\mathcal{M}}

\newcommand{\Oc}{\mathcal{O}}

\newcommand{\Sc}{\mathcal{S}}

\newcommand{\Wc}{\mathcal{W}}

\newcommand{\chisub}[1]{{\chi_{\lower2.5pt\hbox{$\scriptstyle #1$}}}}

\newcommand{\beq}{\begin{equation}}
\newcommand{\eeq}{\end{equation}}

\newcommand{\beqa}{\begin{eqnarray}}
\newcommand{\eeqa}{\end{eqnarray}}

\newcommand{\beqaN}{\begin{eqnarray*}}
\newcommand{\eeqaN}{\end{eqnarray*}}

\renewcommand{\phi}{\varphi}

\renewcommand{\mod}[1]{\hspace{.025in}\left(\text{mod } #1\right)}

\renewcommand{\dim}[1]{\text{dim}_{#1}}
\renewcommand{\ker}[1]{\text{ker}\left(#1\right)}

\newcommand{\crk}[1]{\text{corank}_{#1}}

\renewcommand{\det}[1]{\text{det}\left(#1\right)}
\newcommand{\tr}[1]{\text{tr}\left(#1\right)}

\newcommand{\into}{\hookrightarrow}

\newcommand{\restrict}[1]{\hskip -.05in \upharpoonright_{#1}}

\newcommand{\cc}[1]{\overline{#1}}

\newcommand{\ord}[2]{\text{ord}_{{#1}}{#2}}
\newcommand{\ip}[2]{\langle #1,#2 \rangle}

\newcommand{\Norm}[2]{N_{{#1}/{#2}}}

\newcommand{\ind}[3]{\text{ind}_{{#1}/{#2}}~{#3}}

\newcommand{\pp}{\mathfrak{p}}

\newcommand{\GL}[2]{\text{GL}_{#1}(#2)}

\newcommand{\Hom}[3]{\text{Hom}_{#1}(#2,#3)}
\newcommand{\End}[2]{\text{End}_{#1}(#2)}

\newcommand{\Gal}[2]{\text{Gal}({#1}/{#2})}
\newcommand{\Weil}[2]{\mathcal{W}({#1}/{#2})}
\newcommand{\Selmer}[3]{\text{Sel}_{{#1}}({#2}/{#3})}

\usepackage[all]{xy}
\usepackage{appendix}
\usepackage{url}
\numberwithin{equation}{section}

\title[Local Constants]{Comparing Local Constants of Elliptic Curves in Dihedral Extensions}
\author{Sunil Chetty}
\thanks{This material is based upon work supported by the National Science Foundation under grant DMS-0457481. The author would like to thank Karl Rubin for his many helpful conversations on this material and careful reading of initial drafts of this paper.}

\begin{document}

\begin{abstract}
	In this paper, we study the theories of analytic and arithmetic local constants of elliptic curves, with the work of Rohrlich, for the former, and the work of Mazur and Rubin, for the latter, as a basis. With the Parity Conjecture as motivation, one expects that the arithmetic local constants should be the (algebraic) additive counterparts to ratios of local analytic root numbers. Here, we calculate the constants on both sides in various cases, establishing this connection for a substantial class of elliptic curves. By calculating the arithmetic constants in some new cases, we also extend the class of elliptic curves for which one can determine lower bounds for the growth of $p$-Selmer rank in dihedral extensions of number fields.
\end{abstract}

\maketitle

\section{Introduction}
Let $k$ be a number field. Consider a complex representation $\tau$ of $\Gal{\kb}{k}$ and denote $F=\kb^{\ker{\tau}}$. The Tate module $T_\ell(E)$ of $E$ has a natural $\Gal{\kb}{k}$-action, and this action can be twisted by $\tau$. In turn, defining Euler factors for this twisted Tate module, one obtains a twisted $L$-function $L(E/k,\tau,s)$ associated to $\tau$. Next, defining the usual completed $L$-function (see \cite[\S C.16]{Silv} or \cite[\S 16.3]{Huse})
	$$\Lambda(E/k,\tau,s):=\Norm{k}{\Q}(E,\tau)d_{k/\Q}^2\Gamma_{k}(s)L(E/k,\tau,s),$$
one has the following standard conjecture.

\begin{conjecture}
\label{twistedfunctionaleq}
 The function $\Lambda(E/k,\tau,s)$ satisfies 
 	$$\Lambda(E/k,\tau,s)=W(E/k,\tau) \Lambda(E/k,\cc{\tau},2-s)$$
 where $W(E/k,\tau)=\prod_{u~\text{of}~k} W(E/k_u,\tau)$ and $|W(E/k_u,\tau)|=1$.
\end{conjecture}

Analogously, when $\tau$ is irreducible, one wishes to understand the $\tau$-isotypic part $E(F)^{\tau}$ of the module $E(F)\otimes\C$. For example, when $\tau$ is a 1-dimensional representation, i.e. $\tau:\Gal{F}{k}\rightarrow \Cx$, the $\tau$-part $E(F)^{\tau}$ is
$$E(F)^{\tau}:=\braces{P\in E(F)\otimes\C ~:~ P^\sigma = \tau(\sigma)P~~~\forall \sigma\in \Gal{F}{k}}.$$

An extension (related to the Deligne-Gross Conjecture \cite{RohrRS}) of the Birch and Swinnerton-Dyer Conjecture predicts the following generalization of the Parity Conjecture.

\begin{conjecture}
\label{parityconj}
 If $\tau$ is irreducible and $L(E/k,\tau,s)=L(E/k,\cc{\tau},s)$, then for $r(\tau)=\dim{\C}E(F)^\tau/\dim{}(\tau)$,
	$$W(E/k,\tau)=(-1)^{r(\tau)}.$$
\end{conjecture}

As in \cite{MR}, rather than working with the modules $E(F)^{\tau}$ discussed above, we consider the corresponding modules $\Sc_p(E/F)^{\tau}$ associated to pro-$p$ Selmer groups $\Sc_p(E/F)$ (see \cite[\S 3]{MR}), replacing each $\C$ in the above considerations with $\bar{\Q}_p$. The Shafarevich-Tate Conjecture asserts that $E(F)^{\tau}$ and $\Sc_p(E/F)^{\tau}$ have the same dimension (over their respective fields). Define constants
	$$\begin{array}{rl}
			r_p^{arith}(E,\tau):= & \dim{\Qb_p}\Sc_p(E/F)^\tau/\dim{}(\tau),\\
			r^{an}(E,\tau):= & \ord{s=1}\Lambda(E/k,\tau,s).\\
		\end{array}$$
Conjecture \ref{twistedfunctionaleq} then implies that $W(E/k,\tau)$ determines $r^{an}(E,\tau)\mod{2}$. Considering two representations $\tau$ and $\tau'$, Conjecture \ref{parityconj} and the Shafarevich-Tate Conjecture imply that 
	\begin{equation}
	\label{introrelparity}
		r^{an}(E,\tau)-r^{an}(E,\tau')\equiv r_p^{arith}(E,\tau)-r_p^{arith}(E,\tau')\mod{2}.
	\end{equation}

\subsection{Main Goals}
\label{maingoals}
We now consider a more specific setting. Fix a rational prime $p>3$ and a quadratic extension $K$ of $k$. Next, fix a character $\rho$ of $\Gal{\kb}{K}$ of order $p^n$ and the induced representations $\tau_{\rho}=\ind{K}{k}{\rho}$ and $\tau_1=\ind{K}{k}{1}$ from $\Gal{\kb}{K}$ to $\Gal{\kb}{k}$. With $\rho$ we have an associated cyclic extension $L/K$ of degree $p^n$, specifically $L=\kb^{\ker{\rho}}$, and we assume $\rho$ is such that $L/k$ is Galois. We assume that the tower $k\subset K\subset L$ is dihedral, i.e. that the non-trivial element $c\in\Gal{K}{k}$ acts on $g\in\Gal{L}{k}$ via conjugation as $cgc=g^{-1}$.\\
\indent In this setting Mazur and Rubin \cite{MR} have developed a theory, independent of any conjectures, which gives a method for desribing the right-hand side of \eqref{introrelparity} in terms of a sum over places $v$ of $K$ of arithmetic (cohomological) local constants $\delta_v$ (see \S\ref{arithmeticlocalconstants}). On the analytic side, the decomposition
	$$W(E/k,\tau)=\prod_u W(E/k_u,\tau)$$
of Conjecture \ref{parityconj}, shows that the parity of the left-hand side of \eqref{introrelparity} is given by a sum over places $u$ of $k$ of local constants $\gamma_u\in\Z/2\Z$ defined by
	$$(-1)^{\gamma_u}:=W(E/k_u,\tau_{\rho})/W(E/k_u,\tau_1).$$
It is a natural question to ask whether or not these local constants match ``prime by prime'', i.e. is $\gamma_u\equiv\sum_{v\mid u} \delta_v\mod{2}$?\\
\indent In the following, we calculate these local constants in various cases, in particular extending the calculations of \cite{MR}. Our main result is Theorem \ref{paritythm}, which proves, under mild assumptions, that the analytic and arithmetic local constants do coincide. As an application of the new calculations of the arithmetic local constants, we are able to relax the hypotheses of Theorem 7.2 of \cite{MR} which provides a lower bound for growth in $p$-Selmer rank. As other application, we obtain a different proof (within the assumptions of Theorem \ref{paritythm}) than that of Dokchitser and Dokchitser \cite{DDReg}, and of De La Rochefoucauld \cite{delaRo}, of a relative version of the $p$-Selmer Parity Conjecture.

\section{Local Constants of Elliptic Curves}
The theory describing the analytic local root numbers associated to an elliptic curve $E$ begins with the work of Deligne \cite{Deligne}, Tate \cite{Tate}, and culminates (for our purposes) with Rohrlich \cite{RohrVar}, \cite{RohrWeil}, \cite{RohrGal}. On the algebraic side, Mazur and Rubin \cite{MR} provide the foundational work by defining an arithmetic local constant in terms of a local norm-index of a set of local points of $E$ (see Definition \ref{arithconstdef}). In this section we recall some of the relevant notation, definitions, and results of \cite{RohrGal} and \cite{MR}.

\subsection{Analytic Local Root Numbers}
\label{analyticrootnumber}
 For this section, let $k$ be a local field of residue characteristic $\ell$. Following Rohrlich \cite{RohrGal}, we begin by defining some of the objects which will appear in the next three theorems. We consider a representation $\tau$ of $\Gal{\kb}{k}$ with real-valued character, and we call $W(E/k,\tau)\in\braces{\pm 1}$ the \emph{local analytic root number} for the pair $(E/k,\tau)$. Theorem \ref{Ro2} below is the main tool, and the first claim therein  is most general formulae for $W(E/k,\tau)$ (previously $W(E/k_u,\tau)$ in Conjecture \ref{twistedfunctionaleq}, where $k_u$ denoted the completion of a global $k$). The remainder of the formulae in Theorem \ref{Ro2} are refinements of the first formula.\\ 
\indent Let $\pi$ be the decomposition map associated to $k$ (see \cite[p.126]{RohrWeil}) and define the Weil group of $k$ by
 	$$\Wc(\kb/k):=\pi^{-1}(\langle\Phi\rangle),$$
where $\Phi$ is the Frobenius element $x\mapsto x^q$ and $q$ is the order of the residue field of $k$. As in \cite{RohrGal}, we identify one-dimensional representations of $\Gal{\kb}{k}$ with those of $\kx$ of finite order, via the reciprocity law homomorphism.\\
\indent Let $H$ be the unramified quadratic extension of $k$ and $\eta$ the unramified quadratic character of $\Gal{\kb}{k}$, i.e. the character of $\Gal{\kb}{k}$ with kernel $\Gal{\kb}{H}$. Let $\phi$ be a tamely ramified character of $H^\times$ such that $\phi\restrict{\kx}=\eta$ and $\phi^2\neq 1_H$. For $\sigma=\ind{H}{k}{\phi}$ and $\theta$ the unramified quadratic character of $H^\times$ set $\hat{\sigma}:=\ind{H}{k}{(\phi\theta)}$, which is a dihedral representation of $\kx$ (see p. 316-318 of \cite{RohrGal}). We note that $\hat{\sigma}$ is independent of which tamely ramified character $\phi$ we chose (see p. 329 of \cite{RohrGal}).\\
\indent When $E/k$ has potentially good reduction, define $M$ to be the minimal extension of $k^{ur}$ such that $E$ acquires good reduction over $M$ \cite[\S 2]{ST}, and let the Weil group of $M/k$ be
	$$\Wc(M/k):=\Wc(\kb/k)/\Gal{\kb}{M}.$$
We denote $c_6$ for the constant appearing in a simplified Weierstrauss model for $E/k$ (see \cite[\S III.1]{Silv}). For $\tau$ a representation of $\Gal{\kb}{k}$ with finite image, set $c(\tau):=\det{\tau}(-1)$. Lastly, for two representations $\tau$ and $\tau'$ of $\Gal{\kb}{k}$ define $\ip{\tau}{\tau'}:=\ip{\tr{\tau}}{\tr{\tau'}}$, with the right-hand side the usual inner product on characters.\\
%
%
\indent There is a canonical representation $\sigma_{E/k}$ of $\Wc(\kb/k)$ associated to $E$ (see \cite[p. 329]{RohrGal}), and for $\sigma=\sigma_{E/k}$ Rohrlich proves the following formulae for $W(E/k,\tau)$.
 
\begin{theorem}[Theorems 1 and 2 of \cite{RohrGal}]
\label{Ro2}
  Let $E$ be defined over $k$, and suppose $\text{dim}~\tau=2$. If $\sigma=\ind{H}{k}{\phi}$ then 
  $$W(E/k,\tau)=c(\tau)\cdot(-1)^{\ip{1}{\tau}+\ip{\eta}{\tau}+\ip{\hat{\sigma}}{\tau}}.$$
Moreover,
  \begin{itemize}
   \item[(i)] If $\ell<\infty$ and $E$ has good reduction over $k$ then $W(E/k,\tau)=c(\tau)$.
   \item[(ii)] If $\ell<\infty$ and $v(j)<0$ then
   $$W(E/k,\tau)=c(\tau)(-1)^{\ip{\chi}{\tau}}$$
   where $\chi$ is the character associated to the extension $k(\sqrt{-c_6})$.
   \item[(iii)] If $5\leq\ell<\infty$ and $v(j)\geq 0$,
   $$W(E/k,\tau)=
    \begin{cases} 
     c(\tau) & \text{if $\Wc(M/k)$ is abelian}\\
     c(\tau)(-1)^{\ip{1}{\tau}+\ip{\eta}{\tau}+\ip{\hat{\sigma}}{\tau}} & \text{otherwise}.
    \end{cases}$$
  \end{itemize}
\end{theorem}
 
\begin{remark}
In case (ii) Rohrlich notes that the character $\chi$ is completely determined by the reduction type of $E$. Specifically, $\chi$ is trivial, unramified, or ramified if $E$ has split multiplicative, nonsplit multiplicative, or additive reduction, respectively. In each case $\chi$ is associated to the extension over which $E$ attains split multiplicative reduction (see \cite[p.148]{RohrWeil}). We will make use of this characterization more than the explicit description of the extension.
\end{remark}
 
\begin{prop}[Proposition 7 of \cite{RohrGal}]
\label{Ro7} 
 If $\sigma_{E/k}=\chi\oplus\chi^{-1}$ for some character $\chi$ of $k^\times$ and $\tau$ is as in Theorem \ref{Ro2}, then $W(E/k,\tau)=c(\tau).$
\end{prop}

\subsection{Arithmetic Local Constants}
\label{arithmeticlocalconstants}
As mentioned in the introduction, we will work with pro-$p$ Selmer groups rather than the Mordell-Weil group $E(K)$ of $K$-rational points on $E$. Here we give a brief summary of the standard exact sequences relating $E(K)$ and Selmer groups (see \cite[\S X.4]{Silv} or \cite[\S 8.3]{Huse} for more detail).\\
\indent Our prime $p$ defines an isogeny $E\to E$ and in turn defines the usual $p$-Selmer group sitting in the exact sequence
	$$0\to E(K)/pE(K)\to \Selmer{p}{E}{K}\to \Sha(E/K)[p]\to 0,$$
where $\Sha(E/K)$ is the Shafarevich-Tate group of $E$ over $K$. We have similar sequences for each power $p^i$ of $p$, and in the limit we obtain the $p^\infty$-Selmer group sitting in the exact sequence
	$$0\to E(K)\otimes\left(\Q_p/\Z_p\right)\to \Selmer{p^\infty}{E}{K}\to \Sha(E/K)[p^\infty]\to 0.$$
The pro-$p$ Selmer group of $E$ over $K$ is then the Pontrjagin dual
	$$\Sc_p(E/K):=\Hom{}{\Selmer{p^\infty}{E}{K}}{\Q_p/\Z_p},$$
and we consider it as a $\bar{\Q}_p$-vector space by tensoring with $\bar{\Q}_p$.\\
\indent For the rest of this section we assume the setting and notation of \S\ref{maingoals}. In addition, $v$ will denote a prime of $K$, $u\subset v$ the prime of $k$ below $v$, $w\supset v$ a prime of $L$ above $v$, and denote $k_u$, $K_v$ and $L_w$ for the local fields associated to $u$, $v$, and $w$. We consider $\Gal{L_w}{k_u}\leq\Gal{L}{k}$, and we set $\tau_{\rho,u}$ (resp. $\tau_{1,u}$) to be $\tau_{\rho}$ (resp. $\tau_{1}$) restricted to $\Gal{L_w}{k_u}$.\\
\indent When $L_w\neq K_v$, let $L'_w$ be the unique subfield of $L_w$ containing $K_v$ with $[L_w:L'_w]=p$, and otherwise let $L'_w:=L_w=K_v$.

\begin{definition}[Corollary 5.3 of \cite{MR}]
\label{arithconstdef}
 For each prime $v$ of $K$, define the \emph{arithmetic local constant} $\delta_v:=\delta(v,E,\rho)\in\Z/2\Z$ to be
 $$\delta_v:=\dim{\F_p}E(K_v)/(E(K_v)\cap\Norm{L_w}{L'_w}E(L_w))\mod{2}.$$
\end{definition}

Recall that we have defined (in \S\ref{maingoals}) the ``arithmetic rank'' $r_p^{arith}$ of $(E,\tau)$ by
	$$r_p^{arith}(E,\tau):= \dim{\Qb_p}\Sc_p(E/L)^\tau/\dim{}(\tau).$$
As mentioned in \S\ref{maingoals}, in \cite{MR} Mazur and Rubin show that (a sum of) the $\delta_v$ relate the arithmetic ranks of $(E,\tau_{\rho})$ and $(E,\tau_{1})$. Specifically, they prove the following.  

\begin{theorem}[Theorem 6.4 of \cite{MR}]
\label{MR6.4}
	If $S$ is a set of primes of $K$ containing all primes above $p$, all primes ramified in $L/K$, and all primes where $E$ has bad reduction, then
	$$r_p^{arith}(E,\tau_{\rho})-r_p^{arith}(E,\tau_{1})\equiv\sum_{v\in S}\delta_v~\mod{2}.$$
\end{theorem}
\begin{proof}
	Here we show that the above statement is in fact equivalent to that of Theorem 6.4 of \cite{MR}. The notation follows \cite[\S 3]{MR}. Let $R$ be the maximal order in in the cyclotomic field of $p^n$-roots of unity, so $R$ has a unique prime $\pp$ above $p$. Define $\Ic:=\pp^{p^{n-1}}$ and define the $\Ic$-twist of $E$ by $A:=\Ic\otimes E$ (in the sense of \cite{MRS} and \cite{MR}), an abelian variety with $R\subset\End{K}{A}$. We then have
	$$\begin{array}{l}
			r_p^{arith}(E,\tau_{\rho})=\crk{R\otimes\Z_p}\Selmer{p^\infty}{A}{K},\\
		  r_p^{arith}(E,\tau_{1})=\crk{\Z_p}\Selmer{p^\infty}{E}{K}.
		 \end{array}$$
Thus the conclusion in the theorem is equivalent to
	$$\crk{R\otimes\Z_p}\Selmer{p^\infty}{A}{K}-\crk{Z_p}\Selmer{p^\infty}{E}{K}
				\equiv\sum_{v\in S}\delta_v\mod{2},$$
which is the statement of Theorem 6.4 of \cite{MR}.
\end{proof}

In the next section we will calculate the norm-index defining $\delta_v$ in some cases that were not covered in \cite{MR}. Comparing these calculations in a case by case manner with the analytic root numbers of \S \ref{analyticrootnumber} yields Theorem \ref{paritythm}.\\
\indent Also, in \S\ref{selrankapplications}, we will apply the calculations of the next section in a manner analogous to \S 7 of \cite{MR}. Specifically, we aim to establish a lower bound for the growth of $p$-Selmer rank in a dihedral extension $F/K$ of number fields (see \S\ref{selrankapplications}) with relaxed hypotheses compared to those in the theorems of \S 7 of \cite{MR}.

\section{Local Computations}
\label{localcomputations}
We keep the setting and notation of section \S\ref{arithmeticlocalconstants}. In this section our aim is to calculate the analytic local root numbers and arithmetic local constants in a variety of cases, particularly some not covered in \cite{MR} for the latter. Recall that $c$ is the non-trivial element of $\Gal{K}{k}$.

\subsection{Preliminary Calculations}
\label{splitprimes}
The following two propositions give results for the local analytic root numbers which correspond to Lemma 5.1 and Corollary 5.3 of \cite{MR}. We note that Proposition \ref{analyticMR5.3} allows us to compute $\gamma_u$ independent of the reduction type of $E$ at $v$.

\begin{prop}
\label{analyticMR5.1}
	If $v^c\neq v$, then $W(E/k_u,\tau_{\rho,u})=W(E/k_u,\tau_{1,u})=1.$
\end{prop}
\begin{proof}
	When $v\neq v^c$, we have $K_v=k_u$ and hence $\Gal{L_w}{k_u}=\Gal{L_w}{K_v}$, where $w$ is the prime of $L$ above $v$. It follows that $\tau_{\rho,u}=\rho\oplus\rho^{-1}$ and $\tau_{1,u}=1\oplus 1$. Now, let $\tau=\tau_{\rho,u}$ or $\tau=\tau_{1,u}$. We see $\det{\tau}(-1)=1$ by the decomposition (in either case) of $\tau$. Also
	$$\ip{1}{\tau}+\ip{\eta}{\tau}+\ip{\hat{\sigma}_{E/k}}{\tau}\equiv 0\mod{2},$$
and by the first claim of Theorem \ref{Ro2}, we have $W(E/k_u,\tau)=1$.
\end{proof}

\begin{prop}
\label{analyticMR5.3}
	If $v^c=v$ and $v$ is unramified in $L/K$ then 
	$$W(E/k,\tau_{\rho,u})/W(E/k,\tau_{1,u})=1.$$
\end{prop}
\begin{proof}
	In this case, $v$ splits completely in $L/K$ by \cite[6.5(i)]{MR}, i.e. for every prime $w$ of $L$ lying above $v$,
$L_w=K_v$. Now, we have
	$$\begin{array}{rl}
	  	\tau_{\rho,u}: & \Gal{L_w}{k_u}=\Gal{K_v}{k_u}\to\GL{2}{\C}\\ 
		\tau_{1,u}: & \Gal{L_w}{k_u}=\Gal{K_v}{k_u}\to\GL{2}{\C}
	  \end{array}$$
viewing $\Gal{K_v}{k_u}$ as the $v$-decomposition subgroup of $\Gal{L}{k}$. One sees by direct calculation (see for example \cite[\S5.3]{SerreRep}) that $\tau_{\rho,u}\cong\tau_{1,u}$, and the claim follows.
\end{proof}

\subsection{Good Reduction}
\label{comparinggoodred}
In the case of good reduction, the arithmetic local constant has been determined by Mazur and Rubin in \cite{MR}. We recall those results here.

\begin{theorem}[Theorem 5.6 and 6.6 of \cite{MR}]
\label{MR5.6-6.6}
 If $v$ is a prime of $K$ with $v\nmid p$, $v=v^c$, $v$ is ramified in $L/K$, and $E$ has good reduction at $v$, then
 $\delta_v\equiv 0.$
\end{theorem}

\begin{theorem}[Theorem 5.7 of \cite{MR}]
\label{MR5.7}
 Suppose that $E$ is defined over $\Q_p\subset K_v$ with good supersingular reduction at $p$. If $K_v$ contains the unramified quadratic extension of $\Q_p$, then $\delta_v\equiv 0$.
\end{theorem}

\begin{theorem}[Theorem 6.7 of \cite{MR}]
\label{MR6.7} 
 If $v\mid p$ and $E$ has good ordinary reduction at $v$, then $\delta_v\equiv 0$.
\end{theorem}

For the corresponding situation on the analytic side, we have the following Proposition.

\begin{prop}
\label{analyticgoodred} 
 If $E$ has good reduction over $K_v$ then
 	$$W(E/k_u,\tau_{\rho,u})/W(E/k_u,\tau_{1,u})=1$$
and moreover their common value is 1 when $K_v/k_u$ is unramified.
\end{prop}
\begin{proof} Using part (i) of Theorem \ref{Ro2}, the result is given by the following lemma.\end{proof}

\begin{lemma}
\label{anlemma1} 
 We have $c(\tau_{\rho,u})=c(\tau_{1,u})$ and when $K_v/k_u$ is unramified we have $c(\tau_{\rho,u})=c(\tau_{1,u})=1.$
\end{lemma}
\begin{proof} Notice that $\dim{}\tau_{\rho,u}=\dim{}\tau_{1,u}=2$. We wish to see that $\det{\tau_{\rho,u}}(-1)$ and $\det{\tau_{1,u}}(-1)$ are equal. Since we know that these terms only take the values $\pm 1$, it suffices to see that $\det{\tau_{\rho,u}}\equiv \det{\tau_{1,u}} \mod{\pp}$ for some $\pp\mid p$. Fixing a basis for the spaces of $\rho$ and $1$ respectively, we have that $\rho\equiv 1\mod{\pp}$ from the assumption that $L/K$ is a cyclic $p$-power extension. This implies $\tau_{\rho,u}\equiv\tau_{1,u}\mod{\pp}$, viewed as matrices with function-valued entries and the congruence defined component-wise. It then follows that $\det{\tau_{\rho,u}}\equiv\det{\tau_{1,u}}\mod{\pp}.$\\
\indent Now, let $\tau=\tau_{1,u}$. To determine $\det{\tau}(-1)$, first note that $-1$ is a unit of $k_u$, and when $K_v/k_u$ is unramified, 
$$-1=\Norm{K_v}{k_u}(\beta)\in\Norm{K_v}{k_u}\Oc^{\times}_{K_v}\subset\Norm{K_v}{k_u}\Kx_v .$$
Viewing $\det{\tau}$ as a character of $\kx_u/\Norm{L_w}{k_u}\Lx_w$, we see $\det{\tau}(-1)$ is trivial.
\end{proof}

\subsection{Potential Multiplicative Reduction}
\label{comparingmultred}
In this section, we assume that the prime $v$ of $K$ ramifies in $L/K$ , i.e. that $L_w/K_v$ is non-trivial since $[L_w:K_v]$ is divisible by the ramification degree of $v$ in $L/K$, and satisfies $v^c=v$.

\subsubsection{Analytic}
\begin{prop}
\label{anmult}
If $E/k_u$ has potential multiplicative reduction, then $$W(E/k_u,\tau_{\rho,u})/W(E/k_u,\tau_{1,u})=1$$
if and only if $E$ does not have split multiplicative reduction over $K_v$.
\end{prop}
\begin{proof}
If $E$ has potential multiplicative reduction then the relevant formula is that of Theorem \ref{Ro2} part (ii). Applying Lemma \ref{anlemma1}, it remains to determine $\ip{\chi}{\tau}$.\\
\indent First, if $E$ has split multiplicative reduction at $u$, $\chi$ is the trivial character, and if not then $\chi$ is the character associated to $k_u(\sqrt{-c_6})$. When $\tau=\tau_{\rho,u}$, the assumption $L_w\neq K_v$ implies that $\tau$ is irreducible (see \cite[5.3]{SerreRep}), and since $\dim{}\tau=2$ and $\dim{}\chi=1$, we have that $\tau$ is not isomorphic to $\chi$. Thus, we see that $\ip{\chi}{\tau}=0$.\\
\indent Now consider $\tau=\tau_{1,u}$. As $\tau$ is obtained by inducing the trivial character from $\Gal{L}{K}$ to $\Gal{L}{k}$, it decomposes as $\tau=1\oplus \mu$, with $\mu$ the character associated to the extension $K_v/k_u$. In the case $E$ has split multiplicative reduction at $u$, $\chi=1\not\cong\mu$, by our assumption $v^c=v$, and so $\ip{\chi}{\tau}=1$. Note that in this case, $E$ still has split multiplicative reduction at $v$. For the other cases, $\chi\cong\mu$ if and only if $K_v/k_u$ is the quadratic extension over which $E$ acquires split multiplicative reduction. Combining these, we see that $\ip{\chi}{\tau}=1$ if and only if $E$ has split multiplicative reduction at $v$. Thus, by the formula given in Theorem \ref{Ro2} part (ii), $W(E/k_u,\tau_{1,u})=-1$ if and only if $E$ has split multiplicative reduction at $v$.
\end{proof}

\subsubsection{Arithmetic}
\begin{prop}
\label{arithmult} 
 If $E$ has potential multiplicative reduction over $k_u$, then $\delta_v\equiv 0$ if and only if $E$ does not have split multiplicative reduction over $K_v$.
\end{prop}
\begin{proof}
 Since $E$ has potential multiplicative reduction, $E$ has split multiplicative reduction over $k_u$ or obtains split multiplicative reduction over a quadratic extension $H/k_u$ (see \cite[\S C.14]{Silv}). The latter case splits naturally into $H=K_v$ and $H\neq K_v$. In the case $E$ has split multiplicative reduction over $k_u$, $E$ has split multiplicative reduction over any finite extension of $k_u$ (loc. cit. and \cite[\S VII.5]{Silv}), and so the arguments with $H=K_v$ covers this case.\\
\indent Suppose first that $H=K_v$. In this case there is a $q\in\kx_u$ such that $E(L_w)\cong\Lx_w/q^{\Z}$ as $\Gal{L_w}{K_v}$-modules, and with the isomorphism defined over $K_v$ (loc. cit. \cite{Silv}). This case is Lemma 8.4 of \cite{MR}.\\
\indent Suppose now that $H\neq K_v$. Let $L'_w$ be defined by $K_v\subset L'_w\subset L_w$ with $[L_w:L'_w]=p$. Consider the diagram of fields
 $$
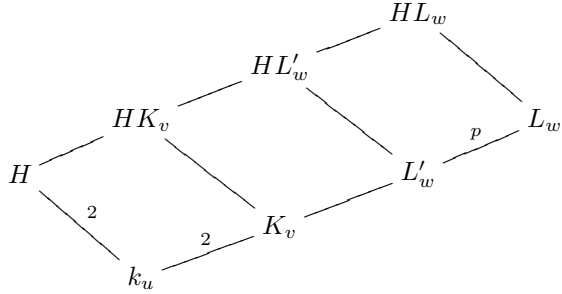
\xymatrix@R=5pt
  {  & & & HL_w\ar@{-}[ddr]\ar@{-}[dl] & \\
     & & HL'_w\ar@{-}[ddr]\ar@{-}[dl] & & \\
     & HK_v\ar@{-}[ddr]\ar@{-}[dl] & & & L_w\ar@{-}[dl]_p\\
   H\ar@{-}[ddr]^2 & & & L'_w\ar@{-}[dl]& \\
     & & K_v\ar@{-}[dl]_2 & & \\
     & k_u & & & & \\
  }$$
and denote $\Delta:=\Gal{H}{k_u}=\langle\sigma\rangle$ and $G=\Gal{HL_w}{HL'_w}=\langle\tau\rangle$. Define $E'$ to be the quadratic twist of $E$ associated to $H/k_u$, so that $E'$ has split multiplicative reduction at $u$, and $E\stackrel{\phi}{\to} E'$ is an isomorphism over $H$. As before, we have a $\Gal{HL_w}{k_u}$-isomorphism $E'(HL_w)\stackrel{\lambda}{\to} H\Lx_w/q^\Z$, with $q\in\kx$. Define the minus-part of $H\Lx_w$ to be 
 $$(H\Lx_w)^-:=\braces{z\in H\Lx_w ~:~ z^\sigma=z^{-1}},$$
and similarly for all other $\Delta$-modules. The map obtained by pre-composing $\lambda$ with $\phi$ restricts to 
 $$E(L_w)\stackrel{\phi}{\longrightarrow} E'(HL_w)^-\stackrel{\lambda}{\longrightarrow}((H\Lx_w)/q^\Z)^-.$$
 If $q\not\in\Norm{HL_w}{L_w}$ then we also have $((H\Lx_w)/q^\Z)^-\cong(H\Lx_w)^-$. If $q\in\Norm{HL_w}{L_w}$ then the projection of $(H\Lx_w)^-$ has index 2 in $((H\Lx_w)/q^\Z)^-$, and so the index is prime to $p$. Both cases will be similar, so we will proceed with the case 
 	$$E(L_w)\to((H\Lx_w)/q^\Z)^-\cong(H\Lx_w)^-.$$
One has a similar situation with $E(L'_w)\to (HL'^{\times}_w)^-.$ From the commutative square
 $$\xymatrix@R=30pt@C=30pt
 	{ E(L_w)\ar[r]^{\sim}\ar[d]^{N} & (H\Lx_w)^-\ar[d]^{N} \\
 	  E(L'_w)\ar[r]^{\sim} & (HL'^{\times}_w)^- \\
 	}$$
 with $N:=\Norm{HL_w}{HL'_w}$, the snake lemma gives 
 $$[E(L'_w):N(E(L_w))]=[(HL'^{\times}_w)^-:N((H\Lx_w)^-)].$$
 We claim that this index is 1, implying $E(K_v)\subseteq E(L'_w)=N(E(L_w))$ and hence
 $$\dim{\F_p}E(K_v)/(E(K_v)\cap\Norm{L_w}{L'_w}E(L_w))=0.$$
 
 It remains to see that $[(HL'^{\times}_w)^-:N((H\Lx_w)^-)]=1.$ With $N$ as above, local class field theory gives an injection
 	$$((HL'_w)^\times)^-/N((H\Lx_w)^-)\into \Gal{HL_w}{HL'_w}=\Gal{L_w}{L'_w}^-.$$
 Since we know that $\sigma$ conjugates $\Gal{L_w}{L'_w}$ trivially, it must be that $\Gal{L_w}{L'_w}^-$ is trivial, finishing the proof.
\end{proof}

\subsection{Potential Good Reduction}
\label{comparingpotgoodred}
In this section, we assume that $v^c=v$ and that $v$ is ramified in $L/K$, so $L_w\neq K_v$ as before.

\subsubsection{Analytic}
\label{potgoodanalytic}
Let $u$, $v$, and $w$ be as in \S\ref{arithmeticlocalconstants}. Denote $\ell$ for their common residue characteristic, and suppose $E/k_u$ has additive and potentially good reduction. Throughout we set $H$ to be the unique unramified quadratic extension of $k_u$.\\
\indent For $\kappa$ a local field, recall $M$ is the minimal extension of $\kappa^{ur}$ over which $E$ acquires good reduction. Denote $\Lambda\cong \Gal{M}{\kappa^{ur}}$ for inertia subgroup (alternatively, $\Lambda$ is the image of the inertia subgroup of $\Wc(\kb/k)$ in $\Wc(M/k)=\Wc(\kb/k)/\Gal{\kb}{M}$).\\
\indent For $\ell\geq 5$, Serre and Tate \cite{ST} prove that $\Lambda=\Gal{M}{k_u^{ur}}$ is always a cyclic group with order dividing 12. From this, Rohrlich is able to prove that the representation $\sigma_{E/k_u}$ attached to $E$ is, after some slight modifications (see p. 332 of \cite{RohrGal}), abelian or induced from a character of the unramified extension $H/k_u$. This property of $\sigma_{E/k_u}$ is essential in Rohrlich's formulae in Theorem \ref{Ro2} for the local root number associated to $E/k_u$ and $\tau$.\\ 
\indent This requirement on $\sigma_{E/k_u}$ is also the main obstruction to applying Rohrlich's techniques to determine the analytic root number over local fields of residue characteristic 2 or 3. The work of Kobayashi \cite{Kob} for $\ell=3$ shows that $\sigma_{E/k_u}$ is \emph{not} in general induced from an unramified quadratic character in certain cases, and hence one cannot apply Theorem \ref{Ro2}. Dokchitser and Dokchitser \cite{DDRes2}, and also Whitehouse \cite{White}, prove that for $\ell=2$ similar obstructions occur. In Proposition \ref{analytic3} below, we are required to make assumptions on $\Lambda$ in order to have a formula for $W(E/k_u,\tau)$. Results of Kraus \cite{Kraus} (see also \cite{White}) give some context to these assumptions, as well as a method of determining $\Lambda$ for a given elliptic curve. 

\begin{prop}
\label{analytic5unram} 
 Suppose $v\nmid 6$. If $v\nmid p$ or $K_v/k_u$ is unramified then $$W(E/k_u,\tau_{\rho,u})/W(E/k_u,\tau_{1,u})=1.$$
\end{prop}
\begin{proof}
 Here, we use the notation of Theorem \ref{Ro2}, and from $v\nmid 6$, we are in the case $\ell\geq 5$. If $v\nmid p$ then Lemma 6.5 of \cite{MR} shows that $K_v/k_u$ is unramified, and if $v\mid p$ then we have assumed $K_v/k_u$ is unramified. As $E$ has potentially good reduction, we have $v(j)\geq 0$, and so part (iii) of Theorem \ref{Ro2} applies. First we notice that $\ip{1}{\tau}+\ip{\eta}{\tau}\equiv 0\mod{2}$. Specifically, $\ip{1}{\tau}=\ip{\eta}{\tau}=0$ when $\tau=\tau_{\rho,u}$ is induced from $\rho$, and since $K_v/k_u$ is unramified, $\ip{1}{\tau}=\ip{\eta}{\tau}=1$ when $\tau=\tau_{1,u}$ is induced from 1. It remains to determine $\ip{\hat{\sigma}}{\tau}$.\\
\indent We know that $\hat{\sigma}$ is the representation of $\Gal{\bar{k}_u}{k_u}$ induced from $\psi:=\phi\theta$, and Rohrlich proves that $\hat{\sigma}$ is dihedral. Recall that $\psi$ is defined to be a character of $H^\times$, and by our assumption that $K_v/k_u$ is unramified, $H^\times=\Kx_v$. If $\psi$ has order $e$, we may view it as a character of $\Gal{K_1}{K_v}$ for some extension $K_1/K_v$ of degree $e$. In turn, we view $\hat{\sigma}$ as a dihedral representation of $\Gal{K_1}{k_u}$.\\
\indent Consider $\tau=\tau_{\rho,u}$. Lifting (\emph{not} inducing) $\hat{\sigma}$ and $\tau$ to some appropriate extension $K_2/k_u$, both are 2-dimensional dihedral representations. Since $\tau$ is irreducible, $\ip{\hat{\sigma}}{\tau}=1$ if and only if $\hat{\sigma}\cong \tau$. Restricting $\tau$ and $\hat{\sigma}$ to $\Gal{K_2}{K_v}$, these representations decompose as $\tau=\rho\oplus\rho^c$ and $\hat{\sigma}=\psi\oplus\psi^c$. The order of $\rho$ is a power of $p\geq 5$ and the order of $\psi$ is 3, 4, or 6 (see \cite[p. 332]{RohrGal}), so $\ip{\hat{\sigma}}{\tau}=0$.\\
\indent For $\tau=\tau_{1,u}$, we have $\tau=1\oplus\eta$ and so $\ip{\hat{\sigma}}{\tau}=0$. Thus, in either case one has
 $$(-1)^{\ip{1}{\tau}+\ip{\eta}{\tau}+\ip{\hat{\sigma}}{\tau}}=(-1)^{\ip{\hat{\sigma}}{\tau}}=1,$$
and applying Lemma \ref{anlemma1} gives $W(E/k_u,\tau_{\rho,u})/W(E/k_u,\tau_{1,u})=1.$
\end{proof}
\begin{prop}
\label{analytic5ram}
 Suppose $v\nmid 6$ and $K_v/k_u$ is ramified. If $E$ acquires good reduction over an abelian extension of $K_v$, then
 	$$W(E/k_u,\tau_{\rho,u})/W(E/k_u,\tau_{1,u})=1.$$
\end{prop}
\begin{proof}
 Here $\ell\geq 5$, so we are in case (iii) of Theorem \ref{Ro2}, and the condition that $E$ acquires good reduction over an abelian extension of $K_v$ is equivalent to (see \cite[Prop 2]{RohrVar}) $\Weil{M}{k_u}$ being abelian. This gives
 $$W(E/k_u,\tau)=c(\tau)=\det{\tau}(-1).$$
 Applying the first claim of Lemma \ref{anlemma1} then gives the result.
\end{proof}

\begin{prop}
\label{analytic3} 
 If $v\mid 6$ and $\Lambda\cong\Z/e\Z$ for $e=1,2,3,4$ or $6$ then $$W(E/k_u,\tau_{\rho,u})/W(E/k_u,\tau_{1,u})=1.$$
\end{prop}
\begin{proof}
 When $v\mid 6$, we have $\ell\leq 3$ and $p\geq 5$, so $v\nmid p$. As in the proof of Proposition \ref{analytic5unram}, we see $K_v/k_u$ is unramified. Proposition 3.2 of \cite{Kob}, when $\ell=3$, and Lemma 5.1 of \cite{White}, when $\ell=2$, give criteria for the image of $\sigma_{E/k,u}$ to be abelian. If it is abelian, then $\sigma_{E/k,u}=\chi\oplus\chi^{-1}$ and so Proposition \ref{Ro7} gives
 $$W(E/k_u,\tau)=\det{\tau}(-1).$$
 Lemma \ref{anlemma1} then gives the desired result.\\
\indent If the image of $\sigma_{E/k,u}$ is not abelian, then we use the assumption that $\Lambda\cong\Z/e\Z$ for $e=2,3,4$ or $6$. In this case, as for $\ell\geq 5$, we have $\sigma:=\sigma_{E/k,u}=\ind{K_v}{k_u}{\chi},$ with $\chi$ a character of $\Kx_v$ (see \cite[\S 2,5]{Kob}, \cite[\S 6]{White}). From the first part of Theorem \ref{Ro2},
 $$W(E/k_u,\tau)=W(\sigma\otimes\tau)=\det{\tau}(-1)\cdot(-1)^{\ip{1}{\tau}+\ip{\eta}{\tau}+\ip{\hat{\sigma}}{\tau}}.$$
Again, $\ip{\hat{\sigma}}{\tau}$ is the significant term. For $\ell\geq 5$ Rohrlich proves that $\hat{\sigma}=\hat{\sigma_e}$, and under the assumption on $\Lambda$ the same proof works when $\ell\leq 3$. So, as above $\ip{\hat{\sigma}}{\tau}=0$ when $\tau=\tau_{\rho,u}$.\\
\indent When $\tau=\tau_{1,u}\cong 1\oplus\eta$, we have $\ip{\hat{\sigma}}{\tau}=0$ and we conclude
 $$W(E/k_u,\tau_{\rho,u})/W(E/k_u,\tau_{1,u})=1.$$
\end{proof}

\subsubsection{Arithmetic}
\label{potgoodarithmetic}
\begin{prop}
\label{arithnotp} 
 Let $K_v$ be a local field, with $v\nmid p$. If $E$ has additive reduction over $K_v$ then $\delta_v=0$.
\end{prop}
\begin{proof}
 If $E$ has additive reduction, then 
 	\begin{equation}
 	\label{additivereductionisom}
 		E_0(K_v)/E_1(K_v)\cong\tilde{E}_{ns}(\kappa)\cong\kappa^+,
 	\end{equation}
 with $\kappa$, the residue field of $K_v$, a finite field of characteristic $\ell\neq p$. We recall two facts (see \S VII.3 and \S VII.6 of \cite{Silv}),
 \begin{enumerate}
  \item[(1)] $E_1(K_v)\cong \Z_{\ell}^r\oplus T$ for some finite $\ell$-group $T$.
  \item[(2)] $|E(K_v)/E_0(K_v)|\leq 4$.
 \end{enumerate}
Since $p\nmid 6\ell$ these two facts yield
	$$E(K_v)/pE(K_v)\cong E_0(K_v)/pE_0(K_v)\cong E_1(K_v)/pE_1(K_v)=0,$$
showing that $E(K_v)$ has no $p$-subgroups and so $\delta_v\equiv 0$.
\end{proof}

The next proposition is a first step into the more difficult case of potentially good reduction at a prime above $p$. To compare with the calculations of the analytic root number, the case that $K_v/k_u$ is ramified (Proposition \ref{analytic5ram}) overlaps, but does not exactly coincide, with this setting (see the proof of Proposition \ref{analytic5unram}). The tools we use for this case were developed by Mazur \cite[\S 4]{MazurRP} and Lubin and Rosen \cite{LubRos} (see also \cite[App. B]{MR}), and as such, we are only able to deal with the case that $E$ attains good ordinary reduction. In \cite[\S 1.d]{MazurRP}, Mazur remarks that the case of supersingular primes is very different and more difficult. Theorem \ref{MR5.7} above remains the only result for $\delta_v$ in the case of supersingular reduction.\\
\indent In order to state the proposition, we must recall a definition (see \cite[App. B]{MR}, also \cite[\S1.b]{MazurRP}). Recall that $E$ is defined over $k_u$, which is an extension of $\Q_p$ when $u\mid p$. For $\Kc$ a finite extension of $k_u$, denote $\kappa$ for the residue field of $\Kc$, and denote $\tilde{E}$ for the reduction of $E$ at the prime of $\Kc$. If $E$ has good ordinary reduction over $\Kc$ then we say that $E$ has \emph{anomalous} reduction over $\Kc$ if $\tilde{E}(\kappa)[p]\neq 0$, and we say $E$ has \emph{non-anomalous} reduction otherwise.

\begin{prop}
\label{arithp}
 If $v\mid p$, $E$ has additive reduction over $K_v$, and $E$ attains good, ordinary, non-anomalous reduction over a Galois extension $M/K_v$ with $[M:K_v]$ prime to $p$, then $\delta_v\equiv 0$.
\end{prop}
\begin{proof} Let $E^k$ denote a model for $E$ defined over $k_u$, and let $E^M$ denote a model of $E$ defined over $M$ for which $E$ has good, ordinary, non-anomalous reduction. We have an isomorphism $E^k\to E^M$ defined over $M$, giving $E^k(\Mc)\cong E^M(\Mc)$, where $\Mc=ML_w$, and similarly for $\Mc'=ML'_w$. We denote $\Gamma:=\Gal{M}{K_v}$ and $H=\Gal{L_w}{L'_w}$, and note that
	$$\Gal{M}{K_v}=\Gal{\Mc'}{L'_w}=\Gal{\Mc}{L_w},\hspace{.15in}\Gal{L_w}{L'_w}=\Gal{\Mc}{\Mc'}.$$
By Propositions B.2 and B.3 of \cite{MR}, we have that $N_H:E^M(\Mc)\to E^M(\Mc')$ is surjective, and hence $N_H:E^k(\Mc)\to E^k(\Mc')$ is surjective also. From this and $N_\Gamma\circ N_H=N_H\circ N_\Gamma$ we have
	\begin{equation}
	\label{arithp-indices}	
		\begin{array}{rl}
			[E^k(L'_w):N_\Gamma(E^k(\Mc'))]= & [E^k(L'_w):N_\Gamma\circ N_H(E^k(\Mc))]\\
			= & [E^k(L'_w):N_H\circ N_\Gamma(E^k(\Mc))].\\
	  \end{array}
	\end{equation}
Since $\Gamma$ has order prime to $p$ and
	$$|\Gamma|\cdot E^k(L'_w)\subset N_\Gamma(E^k(\Mc'))\subset E^k(L'_w),$$
the first term in \eqref{arithp-indices} is prime to $p$. Since $H$ has order $p$ and
	$$N_H\circ N_\Gamma(E^k(\Mc))\subset N_H(E^k(L_w))\subset E^k(L'_w),$$
the last term in \eqref{arithp-indices} is divisible by some power of $p$ when $N_H(E^k(L_w))\neq E^k(L'_w)$. Since this is impossible, we must have 
	$$N_H(E^k(L_w))= E^k(L'_w)\supset E^k(K_v)$$
and hence by Definition \ref{arithconstdef} we have $\delta_v\equiv 0$.
\end{proof}

\section{Results and Applications}
\label{resultsandapplications}
We keep the setting and notation of \S\ref{arithmeticlocalconstants}. Recall that the motivation for the calculations in the previous section has been to determine if the sum $\sum\delta_v$, over $v\mid u$, of arithmetic local constants coincide with a quotient of local root numbers of the conjectural functional equation for $\Lambda(E/k,\tau,s)$. We define the local analytic constant $\gamma_u:=\gamma(u,E,\rho)\in\Z/2\Z$ by
	$$(-1)^{\gamma_u}=W(E/k_u,\tau_{\rho,u})/W(E/k_u,\tau_{1,u}),$$
which gives, analogous to Theorem \ref{MR6.4},
	\begin{equation}
	\label{rootquotient}
		r^{an}(E,\tau_\rho)-r^{an}(E,\tau_1)\equiv\sum_u\gamma_u\mod{2}.
	\end{equation}
With \eqref{introrelparity} in mind, we aim to verify that for every $u$
	\begin{equation}
	\label{arith-analytic-comparison}
			\gamma_u\equiv\sum_{v\mid u}\delta_v\mod{2}.
	\end{equation}

Our main result, Theorem \ref{paritythm} below, gives a setting in which \eqref{arith-analytic-comparison} can be established. Note that for the arithmetic local constant conditions (a) and (b) in the theorem are those addressed in \cite{MR} by Mazur and Rubin (cf. Theorems \ref{MR5.7}-\ref{MR6.7}), and conditions (c) and (d) encompass the new information obtained by the calculations of \S \ref{localcomputations}.\\
\indent Define $\Sk=\braces{\text{primes $v$ of $K$~:~$v^c=v$, $v$ ramifies in $L/K$, and $v\mid 6p$}}.$

\begin{theorem}
\label{paritythm}
 Suppose that for all primes $v\in \Sk$ one of the following holds:
 \begin{enumerate}
   \item[(a)] $E$ has good reduction at $v$ and if $v\mid p$ then $E$ has good ordinary reduction.
   \item[(b)] $E$ is defined over $\Q_p\subset K_v$ with good supersingular reduction at $p$, and $K_v$ contains the unramified quadratic extension of $\Q_p$. 
   \item[(c)] $E$ has potential multiplicative reduction at $v$,
   \item[(d)] $E$ has additive, potential good reduction at $v$, acquiring good reduction over an abelian extension of $K_v$ when $v\mid 6p$, and moreover ordinary, non-anomalous reduction when $v\mid p$.
 \end{enumerate}
 Then for every prime $u$ of $k$, we have $\gamma_u\equiv\sum_{v\mid u}\delta_v\mod{2}$.
\end{theorem}
\begin{proof}
	Let $u$ be a prime of $k$, with $v$, $v^c$ the primes of $K$ above $u$. If $v\not\in\Sk$ then $v^c\neq v$, $v$ is unramified in $L/K$, or $v\nmid 6p$. If $v^c\neq v$ then Lemma 5.1 of \cite{MR} shows that $\delta_v=\delta_{v^c}$ and Proposition \ref{analyticMR5.1} shows that $\gamma_u\equiv 0$, so we have the claim in this case. For the remainder we may assume $v^c=v$, and so we wish to see $\gamma_u\equiv\delta_v$.\\
\indent If $v$ is unramified in $L/K$, Proposition \ref{analyticMR5.3} shows $\gamma_u\equiv 0$. Also, in this case Lemma 6.5 of \cite{MR} shows that $v$ splits completely in $L/K$. Then $\Norm{L_w}{K_v}$ is surjective and hence, by Definition \ref{arithconstdef}, $\delta_v\equiv 0$.\\
\indent In the case $v\nmid 6p$, we have $v\nmid 6$ and $v\nmid p$. If $E$ has good reduction at $v$ then Theorem \ref{MR5.6-6.6} shows $\delta_v\equiv 0$, and Proposition \ref{analyticgoodred} gives $\gamma_u\equiv 0$. If $E$ has potential multiplicative reduction then Proposition \ref{arithmult} and Proposition \ref{anmult}, for $\delta_v$ and $\gamma_u$, respectively, give the result. Lastly, if $E$ has potential good reduction, then we apply Proposition \ref{arithnotp} and Proposition \ref{analytic5unram}.\\
\indent For $v\in\Sk$, cases (a) and (b) follow from Theorem \ref{MR6.7} and Theorem \ref{MR5.7} for $\delta_v$, respectively, and Proposition \ref{analyticgoodred} for $\gamma_u$. Case (c) is covered by Proposition \ref{arithmult} for $\delta_v$ and Proposition \ref{anmult} for $\gamma_u$.\\
\indent For case (d), we consider the (mutually exclusive) cases $v\mid 6$ and $v\mid p$ separately. We first note that the condition that $E$ acquires good reduction over an (necessarily finite) abelian extension of $K_v$ is equivalent to (see \cite[Prop 2]{RohrVar}) $E$ acquiring good reduction over a totally ramified cyclic extension of degree 2, 3, 4 or 6. So, for $v\mid 6$, the conditions of (d) are exactly those necessary to apply Proposition \ref{analytic3} for $\gamma_u$. Also, as $v\nmid p$ we can apply Proposition \ref{arithnotp} for $\delta_v$ and the claim follows. In the latter case, all of the possible degrees for the cyclic extension over which $E$ acquires good reduction are prime to $p\geq 5$. Thus, the extra condition that $E$ acquries ordinary, non-anomalous reduction allows us to apply Proposition \ref{arithp} for $\delta_v$. Also, since $v\nmid 6$ and $v\mid p$, the stated conditions allow us to apply Proposition \ref{analytic5ram} for $\gamma_u$.
\end{proof}

\indent With Theorem \ref{MR6.4} and \eqref{rootquotient} in mind, a consequence of Theorem \ref{paritythm} is the following ``relative'' version of the $p$-Selmer Parity Conjecture (cf. \eqref{introrelparity}). We note that the result is alerady known via work by Dokchitser and Dokchitser in \cite{DDReg}, and can also be recovered from work by Greenberg in \cite{Greenberg}. Our contribution is to obtain a new proof using Mazur and Rubin's theory of arithmetic local constants and to allow for some cases of additive bad reduction at primes above 2 and 3.

\begin{cor}
\label{relparitycor}
 If $E/k$ satisfies the hypothesis of Theorem \ref{paritythm}, then 
 	$$r_p^{arith}(E,\tau_{\rho})-r_p^{arith}(E,\tau_1)\equiv r^{an}(E,\tau_{\rho})-r^{an}(E,\tau_1)\mod{2}.$$
\end{cor}
\begin{proof}
	Applying Theorem \ref{MR6.4} and \eqref{rootquotient} we have $\mod{2}$
		$$\begin{array}{rll}
				r_p^{arith}(E,\tau_{\rho})-r_p^{arith}(E,\tau_1)\equiv & \sum_v\delta_v & \\
				\equiv & \sum_u\gamma_u\equiv & r^{an}(E,\tau_{\rho})-r^{an}(E,\tau_1).
			\end{array}$$
\end{proof}

\subsection{Growth in $p$-Selmer Rank}
\label{selrankapplications}
In addition to the above considerations, the theory of arithmetic local constants allows one to provide lower bounds for growth of $p$-Selmer rank in dihedral extensions of degree $2p^n$. In \S 7 and \S 8 of \cite{MR}, Mazur and Rubin provide the first such theorems, making assumptions on the elliptic curve $E$ in question so that the only $\delta_v$ which survive (in Theorem \ref{MR6.4}) are those which had been calculated in \cite{MR}. Having calculated $\delta_v$ in some new cases of bad reduction, we can make analogous statements with relaxed assumptions on $E$.\\
\indent We now consider $k\subset K\subset F$, where $K/k$ is quadratic as before, $F/K$ is abelian, and $F/k$ a dihedral extension of degree $2p^n$. The cyclic extensions $L/K$ contained in $F$ are in a one-to-one correspondence with the irreducible rational representations $\psi_L$ of $\Gal{F}{K}$. For a fixed $L$, $\rho_L$ is a 1-dimensional component of $\psi_L$, and we apply the calculations of \S\ref{localcomputations} with respect to this $\rho_L$.\\
\indent Let $\Sk=\braces{\text{primes $v$ of $K$~:~$v^c=v$, $v$ ramifies in $F/K$}}$, which differs from that of the previous section by removing the restriction that $v\mid 6p$. Also, define 
	$$\Sk_m:=\braces{\text{$v$ of $K$ for which $E$ has split multiplicative reduction}}.$$
We note that the conditions for $v\mid 6$ of Theorem \ref{paritythm} were imposed in order to handle the local analytic root numbers and so are no longer needed. Before stating the theorem, we recall a result of \cite{MR} for more convenient referencing.

\begin{prop}[Corollary 3.7 of \cite{MR}]
\label{MR3.7}
	There is a $\Gal{F}{K}$-equivariant isomorphism $\Sc_p(E/F)\cong\bigoplus_{L} \Sc_p(E/K)^{\tau_{\rho_L}}$.
\end{prop}

The following is our extension of Theorem 7.2 of \cite{MR}. Again, similar results follow from the work of Dokchitser and Dokchitser in \cite{DDReg} and Greenberg in \cite{Greenberg}.

\begin{theorem}
\label{selrankthm}
	Suppose for each $v\in\Sk$ that $E$ satisfies one of the conditions (a)-(c) of Theorem \ref{paritythm} or
		\begin{itemize}
			\item[(d)] $E$ has additive, potential good reduction at $v$, and if $v\mid p$ then $E$ acquires good ordinary, non-anomalous reduction over a totally ramified Galois extension of $K$ of degree prime to $p$.
		\end{itemize}
	If $\dim{\Q_p}\Sc_p(E/K)+|\Sk_m|$ is odd, then
		$$\dim{\Q_p}\Sc_p(E/F)\geq r_p^{arith}(E,\tau_1)+[F:K]-1.$$
\end{theorem}
\begin{proof}
	The arguments are as in Theorem 8.5 of \cite{MR}, incorporating the new cases in which $\delta_v$ has been determined. By Lemma 5.1 of \cite{MR} and Definition \ref{arithconstdef} (as in the proof of Theorem \ref{paritythm}), we have that $\sum_{v\not\in\Sk}\delta_v\equiv 0\mod{2}$, and so Theorem \ref{MR6.4} gives
		\begin{equation}
		\label{selrankthmMR6.4}
			r_p^{arith}(E,\tau_{\rho})-r_p^{arith}(E,\tau_{1})\equiv\sum_{v\in\Sk} \delta_v.
		\end{equation}
	
By Theorems \ref{MR5.6-6.6}-\ref{MR6.7} and Propositions \ref{arithnotp}, \ref{arithp}, and \ref{arithmult}, we have
		$$\sum_{v\in\Sk} \delta_v=|\Sk_m|.$$		
For each cyclic $L/K$ in $F$ we have a corresponding $R_L$ (see \S\cite{MR}), the maximal order in the cyclotomic field of $[L:K]$-roots of unity, and applying Proposition \ref{MR3.7} we have 
		$$\Sc_p(E/F)\cong\bigoplus_L \Sc_p(E/F)^{\tau_{\rho_L}}\cong\bigoplus_L(R_L\otimes\Q_p)^{d_L}$$
with (except for $L=K$)
	$$d_L\equiv r_p^{arith}(E,\tau_{1})+|\Sk_m|\equiv\dim{\Q_p}\Sc_p(E/F)+|\Sk_m|\mod{2}.$$
We have assumed this to be odd, so $d_L\geq 1$ for each $L\neq K$, and the claim follows.
\end{proof}

\bibliographystyle{plain}
\bibliography{research}

\vspace{0.25in}
\noindent
Department of Mathematics and Computer Science, Colorado College\\
Email: \texttt{sunil.chetty@coloradocollege.edu}

\end{document}